\documentclass{article}
\usepackage[margin = 2cm]{geometry}
\usepackage{epstopdf, braket, amsmath, amsthm, amssymb}
\usepackage[colorlinks = true]{hyperref}

\newtheorem{example}{Example}

\theoremstyle{remark}
\newtheorem{remark}{Remark}

\newtheorem{theorem}{Theorem}
\newtheorem{corollary}{Corollary}
\newtheorem{lemma}{Lemma}
\newtheorem{proposition}{Proposition}

\DeclareMathOperator{\diag}{diag}
\DeclareMathOperator{\trace}{trace}

\title{ON LOG-SUM INEQUALITIES} 
\author{
	Supriyo Dutta \textsuperscript{a}\thanks{CONTACT S. Dutta. Email: dosupriyo@gmail.com, drsupriyo@nita.ac.in} and Shigeru Furuichi \textsuperscript{b}\thanks{CONTACT S. Furuichi. Email: furuichi.shigeru@nihon-u.ac.jp} \vspace{.25 cm}\\ 
	\textsuperscript{a} Department of Mathematics, \\ National Institute of Technology Agartala, \\ Barjala, Jirania, Tripura, India - 799046; \vspace{.25 cm}\\
	\textsuperscript{b}Department of Information Science, \\ College of Humanities and Sciences, Nihon University, \\ 3-25-40, Sakurajyousui, Setagaya-Ku, Tokyo, 156-8550, Japan.
}

\begin{document}

\maketitle

\begin{abstract}
	In information theory, the well-known log-sum inequality is a fundamental tool which indicates the non-negativity for the relative entropy. In this article, we establish a set of inequalities which are similar to the log-sum inequality involving two functions defined on scalars. The parametric extended log-sum inequalities are shown. We extend these inequalities for the commutative matrices.  In addition, utilizing the L\"{o}wner partial order relation and the Hansen-Pedersen theory for non-commutative positive semi-definite matrices we demonstrate a number of matrix-inequalities analogous to the log-sum inequality.
	
\end{abstract}

\section{Introduction}

In information theory, the so-called log-sum inequality is a generalization of Shannon inequality (often also called Klein inequality) which is used to prove the non-negativity of relative entropy. The essence of the non-negativity of the relative entropy is the simple inequality $\ln x \le x-1$ for $x>0$. Therefore, log-sum inequality is important to study information theory. This  is a variant of the Jensen inequality of convex functions, which plays a crucial role for proving the Gibbs' inequality or the convexity of Kullback-Leibler divergence \cite{cover2012elements}.

Recall that, a function $f: X \rightarrow \mathbb{R}$ defined on a convex set $X$ is said to be a convex function if for all $x_1, x_2 \in X$ and for all $t \in [0, 1]$ we have $t f(x_1) + (1 - t) f(x_2) \geq f(tx_1 + (1 - t)x_2)$. When $X$ is the set of real numbers, this inequality is mentioned as the Jensen inequality. In general, we represent the Jensen inequality \cite{jensen1906fonctions}, \cite{niculescu2006convex} as
\begin{equation}\label{Jensen_inequality}
\sum_{i = 1}^n t_i f(x_i) \geq f\left(\sum_{i = 1}^n t_i x_i\right), ~\text{where}~ \sum_{i = 1}^n t_i = 1 ~\text{and}~ 0 \leq t_i \leq 1.
\end{equation}
A twice-differentiable real valued function $f(x)$ on $\mathbb{R}$ is convex if and only if $f''(x) \geq 0$. Let $a_1, a_2, \cdots, a_n$ and $b_1, b_2, \cdots, b_n$ be non-negative numbers. As $f(x) = x\log(x)$ is a convex function, the Jensen inequality \eqref{Jensen_inequality} suggests 
\begin{equation*}
\sum_{i = 1}^n \frac{b_i}{ \sum\limits_{j = 1}^n b_j} \frac{a_i}{b_i}\log \left( \frac{a_i}{b_i} \right) 
\geq \left( \sum_{i = 1}^n \frac{b_i}{ \sum\limits_{j = 1}^n b_j} \frac{a_i}{b_i} \right) \log \left( \sum_{i = 1}^n \frac{b_i}{ \sum\limits_{j = 1}^n b_j} \frac{a_i}{b_i} \right),
\end{equation*}
where $t_i = \dfrac{b_i}{ \sum\limits_{j = 1}^n b_j}$ and $x_i = \dfrac{a_i}{b_i}$. Now, replacing $a = \sum\limits_{j = 1}^n a_j$ and $b = \sum\limits_{j = 1}^n b_j$ we obtain,		
\begin{equation}\label{log_sum_inequality}
\sum_{i=1}^n a_{i}\log \left(\frac {a_{i}}{b_{i}} \right) \geq a\log \left(\frac {a}{b} \right),
\end{equation}  
which is the standard log-sum inequality. The inequality is valid for $n = \infty$ provided $a < \infty$ as well as $b < \infty$. In \cite{csiszar2004information}, an analog of log-sum inequality is derived as 
\begin{equation*}
\sum_{i = 1}^n b_i f \left( \frac{a_i}{b_i} \right) \geq b f \left( \frac{a}{b} \right),
\end{equation*}
where $f$ is strictly convex at $\dfrac{a}{b}$. The equality holds if and only if $a_i = c b_i$, for $i = 1, 2, \cdots n$, and constant number $c$.

Although log sum inequality is important in classical information theory, a detailed study on this topic is not available in the literature, to the best of our knowledge. This article demonstrates a number of inequalities which are analogous to the log-sum inequality. First, we extend the log-sum inequality with two real-valued functions, which is illustrated in Theorem \ref{log_sum_inequality_with_real_functions}. This theorem expands the applicability of the log-sum inequality as it breaks the restriction that $a_i$ and $b_i$ are the real numbers. The convexity of a function is essential for proving log-sum inequality. Considering the concavity of a function, we also find analogous results. Investigating the variants of log-sum inequality for concave functions is crucial as there are classes of functions whose convexity depends on specific values of a parameter, for instance, the $q$-deformed logarithm. We also elaborate the matrix analogs of the log-sum inequality. Here we observe that all the inequalities derived for the real functions can be easily extended as trace-form-inequality for commuting self-adjoint matrices. A recent trend in matrix analysis is constructing matrix theoretic counterparts of the known inequalities for real functions, where the matrix inequality is driven by the L\"{o}wner partial order relation. Proving the log-sum inequality in this context is non-trivial and difficult. We are able to derive it under several conditions. To the best of our knowledge, this is the first attempt in literature to extend the log-sum inequality for the non-commutative positive semidefinite matrices.

This article is distributed into five sections. Preliminary concepts are discussed before applying them in a mathematical derivation. In section 2, we consider inequalities for real-valued functions. Here, we discuss the log-sum inequality for deformed logarithms, for instance the $q$-deformed logarithm. Section 3 is dedicated to the discussion of log-sum inequality as a trace-form-inequality for commuting self-adjoint matrices. It has a number of consequences in quantum information theory. In the next section, we attempt to derive the log-sum inequality for L\"{o}wner partial order relation. This section extensively uses the idea of operator monotone, operator convex functions, and the operator Jensen inequality, which we mention at the beginning of section 4. Then we conclude the article.

	\section{log-sum inequalities for real functions}\label{Generalized_log_sum_inequality_for_real_functions}

In this section, we extend the idea of the log-sum inequality for real-valued functions. Both concave and convex functions are considered for the investigation. Interestingly, the $q$-deformed logarithm is convex when $q < 2$, which assists us to illustrate two sets of log-sum inequalities. We begin our discussion with the following theorem.

\begin{theorem}\label{log_sum_inequality_with_real_functions}
	Let $g$ be a real valued function whose domain contains $a_1, a_2, \cdots, a_n$ and $b_1, b_2, \cdots , b_n$, such that $g(b_i) > 0$ for all $i = 1, 2, \cdots n$. Consider another function $f : \left[m_g, M_g \right] \rightarrow \mathbb{R}$ for which $h(x) = x f(x)$ is convex, where $m_g:=\min\limits_i \left\{ \dfrac{g(a_i)}{g(b_i)} \right\}$ and $M_g:=\max\limits_i \left\{ \dfrac{g(a_i)}{g(b_i)} \right\}$. Then,
	\begin{equation}\label{theorem01_eq01}
	\sum_{i = 1}^n g(a_i) f\left(\frac{g(a_i)}{g(b_i)} \right) \geq \left( \sum_{i = 1}^n g(a_i) \right) f \left(\frac{\sum_{i = 1}^n g(a_i) }{\sum_{i = 1}^n g(b_i)}\right).
	\end{equation}	
\end{theorem}	

\begin{proof}
	Define $a = \sum\limits_{i = 1}^n g(a_i)$ and $b = \sum\limits_{i = 1}^n g(b_i) > 0$. Note that, $\dfrac{a}{b} = \sum\limits_{i = 1}^n \dfrac{g(b_i)}{b} \dfrac{g(a_i)}{g(b_i)}$, which is a convex combination of $\dfrac{g(a_i)}{g(b_i)}$ for $i = 1, 2, \cdots n$. Clearly, the ratio $\dfrac{a}{b}$ belongs to the convex set $\left[m_g, M_g \right]$. Now,
	\begin{equation*}
	\sum_{i = 1}^n g(a_i) f\left(\frac{g(a_i)}{g(b_i)} \right) = \sum_{i = 1}^n g(b_i) \frac{g(a_i)}{g(b_i)} f\left(\frac{g(a_i)}{g(b_i)} \right) = \sum_{i = 1}^n b \frac{g(b_i)}{b} h \left(\frac{g(a_i)}{g(b_i)}\right).
	\end{equation*}
	As $g(b_i) > 0$ for all $i = \{1, 2, \cdots n\}$ and $b = \sum\limits_{i = 1}^n g(b_i)$ we have $0 \leq \dfrac{g(b_i)}{b} \leq 1$ and $\sum\limits_{i = 1}^n \dfrac{g(b_i)}{b} = 1$. Now, the Jensen inequality indicates
	\begin{equation}\label{log_sum_inequality_with_real_functions_2nd_equn}
	\begin{split}  
	\sum_{i = 1}^n b \frac{g(b_i)}{b} h \left(\frac{g(a_i)}{g(b_i)}\right) \geq b h \left(\sum_{i = 1}^n \frac{g(b_i)}{b} \frac{g(a_i)}{g(b_i)} \right) = b h \left(\frac{1}{b} \sum_{i = 1}^n g(a_i) \right) = b h \left(\frac{a}{b}\right).
	\end{split}
	\end{equation}
	Expanding $h$ we get 
	\begin{equation*}
	b h \left(\frac{a}{b}\right) = b \frac{a}{b} f \left(\frac{a}{b} \right) = a f \left(\frac{a}{b} \right) = \left( \sum_{i = 1}^n g(a_i) \right) f \left(\frac{\sum\limits_{i = 1}^n g(a_i) }{\sum\limits_{i = 1}^n g(b_i)}\right).
	\end{equation*} 
	Combining, we obtain the result.
\end{proof}

From now on, we denote $m_g:=\min\limits_i \left\{ \dfrac{g(a_i)}{g(b_i)} \right\}$ and $M_g:=\max\limits_i \left\{ \dfrac{g(a_i)}{g(b_i)} \right\}$ throughout the article, where $g$ is a real valued function whose domain of definition contains $a_1, a_2, \cdots, a_n$ and $b_1, b_2, \cdots , b_n$, such that $g(b_i) > 0$ for all $i = 1, 2, \cdots n$.

Defining $g(x) = x$ and $f(x) = \log(x)$ for $x \in \mathbb{R}_{> 0}$, we observe that $h(x) = x\log(x)$ is a convex function. Then, Theorem \ref{log_sum_inequality_with_real_functions} leads us to the log-sum inequality mentioned in equation (\ref{log_sum_inequality}). 

Theorem \ref{log_sum_inequality_with_real_functions} expands the scope of applications of the log-sum inequality. The domain of $g$ may not be a convex set. We can consider $a_i$ and $b_i$ arbitrarily such that $g(b_i) > 0$ for all $i = 1, 2, \cdots n$. In other words, the range of $g$ must have a non-empty intersection with the set of positive reals.

Let $f$ be a twice-differentiable function in $\left[m_g, M_g \right]$. The function $xf(x)$ is convex if and only if
\begin{equation}\label{2nd_deriv_eq01}
\frac{d^2}{dx^2}(x f(x)) = x f''(x) + 2 f'(x) \geq 0.
\end{equation}
If $m_g \geq 0$, then any monotone increasing and convex function $f$ defined on $\left[m_g, M_g \right]$ fulfill equation (\ref{2nd_deriv_eq01}). Now, we are in a position to consider a few special cases of Theorem \ref{log_sum_inequality_with_real_functions}.

\begin{example} 
	Define $g(x) = x^r$ for some real parameter $r$, and $f(x) = \log(x)$ for $x > 0$. As $h(x) = x \log(x)$ is a convex function applying Theorem \ref{log_sum_inequality_with_real_functions} we observe
	\begin{equation*} 
	\sum_{i = 1}^n a_i^r \log \left(\frac{a_i^r}{b_i^r}\right) \geq \left(\sum_{i = 1}^n a_i^r\right)\log \left(\frac{\sum\limits_{i = 1}^n a_i^r}{\sum\limits_{i = 1}^n b_i^r}\right). 
	\end{equation*} 	
	Thus we have
	\begin{equation}\label{function_inequality_with x_r_and_log_x}
	\sum_{i = 1}^n ra_i^r \left\{ \log(a_i) - \log(b_i)\right\} \geq \left( \sum_{i = 1}^n a_i^r \right) \left\{ \log \left( \sum_{i = 1}^n a_i^r \right) - \log \left( \sum_{i = 1}^n b_i^r \right) \right\}.
	\end{equation}
\end{example}

\begin{example} 
	The $q$-deformed logarithm or $q$-logarithm \cite{kac2001quantum} is defined by
	\begin{equation}\label{q_deformed_logarithm}
	f(x) = \ln_q(x) = \frac{x^{1 - q} - 1}{1 - q},
	\end{equation}
	for $x > 0$ and $q \neq 1$. Note that $xf(x)$ is a convex function for $q < 2$ as $\dfrac{d^2}{dx^2} \left(x f(x) \right) = \dfrac{2 - q}{x^q} \geq 0$. The division rule of $q$-deformed logarithm can be expressed as
	\begin{eqnarray}
		\ln_q \left(\frac{x}{y}\right) &=& \ln_q(x) + \ln_q \left(\frac{1}{y}\right) + (1 - q) \ln_q(x) \ln_q \left(\frac{1}{y}\right)\nonumber  \\
		&=& \frac{\ln_q(x) - \ln_q(y)}{1 + (1 - q) \ln_q(y)} = \frac{\ln_q(x) - \ln_q(y)}{y^{1 - q}}.\label{q_log_division_rule}
	\end{eqnarray}
	Putting $g(x) = x^r$ for $x>0$ and real parameter $r$ as well as $f(x) = \ln_q(x)$ with $q < 2$ in Theorem \ref{log_sum_inequality_with_real_functions} we obtain 
	\begin{equation*}
	\sum_{i = 1}^n a_i^r \ln_q \left(\frac{a_i^r}{b_i^r}\right) \geq \left(\sum_{i = 1}^n a_i^r\right) \ln_q \left(\frac{\sum\limits_{i = 1}^n a_i^r}{\sum\limits_{i = 1}^n b_i^r}\right),
	\end{equation*}
	where $a_1, a_2, \cdots a_n$ and $b_1, b_2, \cdots b_n$ are positive real numbers. Now applying equation (\ref{q_log_division_rule}),  we find
	\begin{equation*}
	\ln_q \left(\frac{\sum\limits_{i = 1}^n a_i^r}{\sum\limits_{i = 1}^n b_i^r}\right) = \frac{\ln_q \left(\sum\limits_{i = 1}^n a_i^r\right) - \ln_q \left(\sum\limits_{i = 1}^n b_i^r\right)}{\left(\sum\limits_{i = 1}^n b_i^r\right)^{1 - q}}.
	\end{equation*} 
	Combining we get
	\begin{equation}\label{q_deformed_log_sum_with_r}
	\left(\sum_{i = 1}^n b_i^r\right)^{1 - q} \sum_{i = 1}^n a_i^r \ln_q \left(\frac{a_i^r}{b_i^r}\right) \geq \left(\sum_{i = 1}^n a_i^r\right) \left\{\ln_q \left(\sum_{i = 1}^n a_i^r\right) - \ln_q \left(\sum_{i = 1}^n b_i^r\right) \right\}.
	\end{equation}
	For $r = 1$ the above inequality reduces to
	\begin{equation}\label{q_deformed_log_sum_with_r_is_1}
	\left(\sum_{i = 1}^n b_i\right)^{1 - q} \sum_{i = 1}^n a_i \ln_q \left(\frac{a_i}{b_i}\right) \geq \left(\sum_{i = 1}^n a_i \right) \left\{ \ln_q \left(\sum_{i = 1}^n a_i\right) - \ln_q \left(\sum_{i = 1}^n b_i\right) \right\}.
	\end{equation}
	Instead of $g(x) = x^r$, one may consider trigonometric, exponential, hyperbolic, or any other function to get new inequalities. 
\end{example}

Recall that, a real-valued function $f$ defined on a convex set $X$ is said to be concave if $-f(x)$ is convex. Applying it in equation (\ref{Jensen_inequality}) observe that for a concave function $f$ we have
\begin{equation*}
\sum_{i = 1}^n t_i f(x_i) \leq f\left(\sum_{i = 1}^n t_i x_i\right), ~\text{where}~ \sum_{i = 1}^n t_i = 1 ~\text{and}~ 0 \leq t_i \leq 1.
\end{equation*}
Considering $h(x) = xf(x)$ as a concave function in equation (\ref{log_sum_inequality_with_real_functions_2nd_equn}), we obtain
\begin{equation*}
\sum_{i = 1}^n g(a_i) f\left(\frac{g(a_i)}{g(b_i)} \right) \leq \left( \sum_{i = 1}^n g(a_i) \right) f \left(\frac{\sum\limits_{i = 1}^n g(a_i) }{\sum\limits_{i = 1}^n g(b_i)}\right),
\end{equation*}
under the equivalent conditions on $a_i$ and $b_i$ as well as the real valued functions $f$ and $g$ as mentioned in Theorem \ref{log_sum_inequality_with_real_functions}.

When $q > 2$ in equation (\ref{q_deformed_logarithm}) we observe that $h(x) = x f(x)$ is a concave function. Therefore the inequalities in equation (\ref{q_deformed_log_sum_with_r}) and (\ref{q_deformed_log_sum_with_r_is_1}) becomes
\begin{equation}\label{q_deformed_log_sum_with_r_and_q_>_2}
\left(\sum_{i = 1}^n b_i^r\right)^{1 - q} \sum_{i = 1}^n a_i^r \ln_q \left(\frac{a_i^r}{b_i^r}\right) \leq \left(\sum_{i = 1}^n a_i^r\right) \left\{\ln_q \left(\sum_{i = 1}^n a_i^r\right) - \ln_q \left(\sum_{i = 1}^n b_i^r\right) \right\}.
\end{equation}
and
\begin{equation*}
\left(\sum_{i = 1}^n b_i\right)^{1 - q} \sum_{i = 1}^n a_i \ln_q \left(\frac{a_i}{b_i}\right) \leq \left(\sum_{i = 1}^n a_i \right) \left\{ \ln_q \left(\sum_{i = 1}^n a_i\right) - \ln_q \left(\sum_{i = 1}^n b_i\right) \right\},
\end{equation*}
respectively, under the similar conditions on $a_i$ and $b_i$. 

If $f(x) = \log(x)$ we find that $-xf(x) = x f\left(\dfrac{1}{x}\right)$ is a concave function. The equation (\ref{log_sum_inequality}) suggests that
\begin{equation*}
- \sum_{i=1}^n a_{i}\log \left(\frac {a_{i}}{b_{i}} \right) \leq - a\log \left(\frac{a}{b} \right)
\end{equation*}
which implies
\begin{equation*}
\sum_{i=1}^n a_{i}\log \left(\frac {b_{i}}{a_{i}} \right) \leq  a\log \left(\frac{b}{a} \right).
\end{equation*} 
In general, $-xf(x) = x f\left(\dfrac{1}{x}\right)$ does not hold. For instance, consider $f(x) = \ln_q(x)$, which refers $-x \ln_q(x) = x^q \ln_q \left(\dfrac{1}{x}\right)$. This fact leads us to a new inequality which we consider in the following theorem.

\begin{theorem}\label{reverse_log_sum_inequality_with_real_functions}
	Let $g$ be a real-valued function whose domain contains $a_1, a_2, \cdots, a_n$ and $b_1, b_2, \cdots , b_n$, such that $g(a_i) > 0$ for $i = 1, 2, \cdots n$; and $f : \left[\widetilde{m}_g, \widetilde{M}_g \right] \rightarrow \mathbb{R}$ be a function for which $h(x) = x f\left( \frac{1}{x} \right)$ is a concave function, where $\widetilde{m}_g = \min\limits_{1\le i\le n} \left\{\dfrac{g(b_i)}{g(a_i)} \right\}$, and $\widetilde{M}_g = \max\limits_{1\le i\le n} \left\{ \dfrac{g(b_i)}{g(a_i)} \right\}$. Then,
	$$\sum_{i = 1}^n g(a_i) f\left(\frac{g(b_i)}{g(a_i)} \right) \leq \left( \sum_{i = 1}^n g(a_i) \right) f \left(\frac{\sum\limits_{i = 1}^n g(b_i) }{\sum\limits_{i = 1}^n g(a_i)}\right).$$
\end{theorem}

\begin{proof}
	It is easy to find that,
	\begin{equation*}
	\begin{split}
	& \sum_{i = 1}^n g(a_i) f\left(\frac{g(b_i)}{g(a_i)} \right) = \sum_{i = 1}^n g(b_i) \frac{g(a_i)}{g(b_i)} f\left(\frac{g(b_i)}{g(a_i)} \right) = \sum_{i = 1}^n b \frac{g(b_i)}{b} h \left(\frac{g(a_i)}{g(b_i)}\right)\\
	\leq & b h \left(\sum_{i = 1}^n \frac{g(b_i)}{b} \frac{g(a_i)}{g(b_i)} \right) \quad \left[\text{as}~ h(x) = x f \left(\frac{1}{x} \right) ~\text{is concave}\right]\\
	= & b h \left(\frac{1}{b} \sum_{i = 1}^n g(a_i) \right) = b h \left(\frac{a}{b}\right) = b \frac{a}{b} f \left(\frac{b}{a} \right) = a f \left(\frac{b}{a} \right)\\
	\leq & \left( \sum_{i = 1}^n g(a_i) \right) f \left(\frac{\sum\limits_{i = 1}^n g(b_i) }{\sum\limits_{i = 1}^n g(a_i)}\right).
	\end{split}
	\end{equation*}
\end{proof}	

From now on, we use $\widetilde{m}_g = \min\limits_{1\le i\le n} \left\{\dfrac{g(b_i)}{g(a_i)} \right\}$, and $\widetilde{M}_g = \max\limits_{1\le i\le n} \left\{ \dfrac{g(b_i)}{g(a_i)} \right\}$ throughout this article, where $g$ is a real valued function whose domain of definition contains $a_1, a_2, \cdots, a_n$ and $b_1, b_2, \cdots , b_n$, such that $g(a_i) > 0$ for all $i = 1, 2, \cdots n$.

We know that a twice-differentiable function $h(x)$ is concave if and only if $h''(x) \le 0$, which indicates 
\begin{equation}\label{2nd_deriv_eq02}
\frac{d^2}{dx^2} \left( x f\left(\frac{1}{x} \right) \right) = \frac{1}{x^3}f''\left(\frac{1}{x} \right) \leq 0.
\end{equation}
Note that, equation (\ref{2nd_deriv_eq02}) and (\ref{2nd_deriv_eq01}) are not equivalent, in general. For example, consider the function $f(x) = \frac{x}{x^2+2}$, where $x > 0$. We can calculate
\begin{equation*}
f'(x) = \frac{2-x^2}{\left(x^2+2\right)^2}, ~\text{and}~ f''(x) = \frac{2 x \left(x^2-6\right)}{\left(x^2+2\right)^3}.
\end{equation*}
Now, equation (\ref{2nd_deriv_eq02}) suggests that $xf\left(\dfrac{1}{x}\right)$ is concave when 
\begin{equation*}
\frac{1}{x^3}f''\left(\frac{1}{x} \right) = \frac{2x^2(1 - 6 x^2)}{\left(2 x^2+1\right)^3} \leq 0 ~\text{that is} ~ x \geq \frac{1}{\sqrt{6}}.
\end{equation*}
But, putting the values of $f'$ and $f''$ in equation (\ref{2nd_deriv_eq01}) we observe that $xf(x)$ is concave when
\begin{equation*} 
x f''(x) + 2 f'(x) = \frac{8-12 x^2}{\left(x^2+2\right)^3} \leq 0 ~\text{when}~ x \geq \sqrt{\frac{2}{3}}.
\end{equation*} 
Now, replacing $f(x) = \dfrac{x}{x^2+2}$ and $g(x) = x$ in Theorem \ref{reverse_log_sum_inequality_with_real_functions} we observe
\begin{equation*}
\sum_{i = 1}^n \frac{a_i^2 b_i}{2 a_i^2 + b_i^2} \leq \frac{\left(\sum\limits_{i = 1}^n a_i\right)^2 \left(\sum\limits_{i = 1}^n b_i\right)}{2 \left(\sum\limits_{i = 1}^n a_i\right)^2 + \left(\sum\limits_{i = 1}^n b_i\right)^2},
\end{equation*}	
where $a_i$ and $b_i$ are greater than $\dfrac{1}{\sqrt{6}}$.

\section{log-sum inequalities for commuting self-adjoint matrices}

We begin this section with a number of basic concepts of matrix analysis. Given any self-adjoint matrix $A$, there exists a unitary matrix $U$, such that, $A = U \Lambda(A) U^\dagger$, where $\Lambda(A) = \diag\{a_1, a_2, \cdots, a_n\}$ is a diagonal matrix whose diagonal entries, $a_i$ for $i = 1, 2, \cdots, n$, are the eigenvalues of $A$. Throughout this article $U^\dagger$ denotes the conjugate transpose of the matrix $U$. If the matrices $A$ and $B$ commute, that is $AB = BA$, then there exists a unitary matrix $U$, such that, $A = U \Lambda(A) U^\dagger$ and $B = U \Lambda(B) U^\dagger$, hold simultaneously. We also denote $a = \trace(A) = \sum\limits_{i = 1}^n a_i$. Let $f$ be a continuous real valued function defined on an interval $J$ and $A$ be a self-adjoint matrix with eigenvalues in $J$, then
\begin{equation}\label{matrix_function}
f(A) = U \diag\{f(a_i): i = 1, 2, \cdots, n\} U^\dagger.
\end{equation}
Now, we demonstrate the results derived in Section \ref{Generalized_log_sum_inequality_for_real_functions} for the self-adjoint matrices.

\begin{theorem}\label{log_sum_inequality_with_matrices}
	Let $A$ and $B$ be commuting self-adjoint matrices, with the sets of eigenvalues $\Lambda(A) = \{a_1, a_2, \cdots , a_n\}$ and $\Lambda(B) = \{b_1, b_2, \cdots, b_n\}$, respectively. Also, let $g: \Lambda(A) \cup \Lambda(B) \rightarrow \mathbb{R}$ be a function, such that, $g(b_i) > 0$ for all $i = 1, 2, \cdots n$. In addition, $f : \left[m_g, M_g \right] \rightarrow \mathbb{R}$ is a function for which $xf(x)$ is convex. Then,
	$$\trace[g(A) f(g(A)(g(B))^{-1})] \geq \trace[g(A)] f \left(\frac{\trace(g(A))}{\trace(g(B))}\right).$$
\end{theorem}
\begin{proof}
	As $A$ and $B$ are commuting self-adjoint matrices there exists an unitary matrix $U$, such that,
	\begin{equation*}
	\begin{split}
	& g(A) = U \diag\{g(a_i): i = 1, 2, \cdots, n\} U^\dagger, \\
	& g(B) = U \diag\{g(b_i): i = 1, 2, \cdots, n\} U^\dagger, \\
	\text{and}~ & (g(B))^{-1} = U \diag\left\{\frac{1}{g(b_i)}: i = 1, 2, \cdots, n\right\} U^\dagger,
	\end{split}
	\end{equation*} 
	as $g(b_i) > 0$ for all $i = 1, 2, \cdots n$. Therefore, 
	\begin{equation*}
	f(g(A)(g(B))^{-1}) = U \diag\left\{f \left( \frac{g(a_i)}{g(b_i)} \right): i = 1, 2, \cdots, n\right\} U^\dagger,
	\end{equation*}
	which indicates
	\begin{equation*}
	g(A) f(g(A)(g(B))^{-1}) = U \diag\left\{g(a_i) f \left( \frac{g(a_i)}{g(b_i)} \right): i = 1, 2, \cdots n \right\} U^\dagger.
	\end{equation*}
	Note that, $g(a_i) f \left( \dfrac{g(a_i)}{g(b_i)} \right)$ are the eigenvalues of the matrix $g(A) f(g(A)(g(B))^{-1})$ for $i = 1, 2, \cdots n$. Hence,
	\begin{equation*}
	\trace(g(A) f(g(A)(g(B))^{-1})) = \sum_{i = 1}^n g(a_i) f \left( \frac{g(a_i)}{g(b_i)} \right).
	\end{equation*} 
	Also, $\sum\limits_{i = 1}^n g(a_i) = \trace(g(A))$ and $\sum\limits_{i = 1}^n g(b_i) = \trace(g(B))$. Applying Theorem \ref{log_sum_inequality_with_real_functions} we obtain
	\begin{equation*}
		\sum_{i = 1}^n g(a_i) f \left( \frac{g(a_i)}{g(b_i)} \right) = \left( \sum_{i = 1}^n g(a_i) \right) f \left(\frac{\sum\limits_{i = 1}^n g(a_i) }{\sum\limits_{i = 1}^n g(b_i)}\right) = \trace[g(A)] f \left(\frac{\trace(g(A))}{\trace(g(B))}\right).
	\end{equation*}
	Combining we get the result.
\end{proof}

The following corollary holds trivially from the above theorem. 
\begin{corollary}
	Given positive definite, commuting matrices $A$ and $B$,
	$$\trace(\exp(A \log(A))) - \trace(\exp(A \log(B))) \geq \trace(A) \log \left( \frac{\trace(A)}{\trace(B)} \right).$$
\end{corollary}
\begin{proof}
	Consider $f(x) = \log(x)$ and $g(x) = x$ in Theorem \ref{log_sum_inequality_with_matrices}. As $A$ and $B$ are positive definite we have $a_i > 0$ and $b_i > 0$ for all $i = 1, 2, \cdots, n$. Note that, 
	\begin{equation*}
		\sum_{i=1}^n a_{i}\log {\frac {a_{i}}{b_{i}}} = \sum_{i = 1}^n \log \left( a_i^{a_i} \right) - \sum_{i = 1}^n \log \left( b_i^{a_i} \right) = \trace(\exp(A \log(A))) - \trace(\exp(A \log(B))).
	\end{equation*}
	Also, applying $\sum\limits_{i = 1}^n a_i = \trace(A)$ and $\sum\limits_{i = 1}^n b_i = \trace(B)$ we observe that 
	$$\trace[g(A)] f \left(\dfrac{\trace(g(A))}{\trace(g(B))}\right) = \trace(A) \log \left( \dfrac{\trace(A)}{\trace(B)} \right).$$ 
	Combining we get the result.
\end{proof}

Theorem \ref{log_sum_inequality_with_matrices} is a matrix-theoretic counterpart of Theorem \ref{log_sum_inequality_with_real_functions}, which assists us to establish a number of matrix inequalities, which are derived in Section \ref{Generalized_log_sum_inequality_for_real_functions} for real functions. Some of these matrix inequalities have immediate consequences in quantum information theory. Consider the following example.

\begin{example} 
	The matrix counterpart of equation (\ref{function_inequality_with x_r_and_log_x}) is represented as
	\begin{equation}\label{matrix_inequality_with x_r_and_log_x} 
	\trace \left[A^r \log(A^rB^{-r}) \right] \geq \trace(A^r) [\log(\trace(A^r)) - \log(\trace(B^r))],
	\end{equation}
	where $A$ and $B$ are self-adjoint commuting matrices as well as $B$ is positive definite. Putting $r = 1$ in the above inequality we have
	\begin{equation*}
	\trace \left[A \log(AB^{-1}) \right] \geq \trace(A) [\log(\trace(A)) - \log(\trace(B))].
	\end{equation*} 
	Considering $AB = BA$, we can prove that $\log(AB^{-1}) = \log(A) - \log(B)$. Replacing it at the left hand side \cite{bebiano2003matrix}, \cite[Theorem 3.3]{furuichi2004fundamental}
	\begin{equation}\label{inequality_bebiano_2003_matrix}
	\trace \left[A (\log(A) - \log(B)) \right] \geq \trace(A) [\log(\trace(A)) - \log(\trace(B))].
	\end{equation}
\end{example} 

Recall that, in quantum information theory \cite{wilde2013quantum} a quantum state is represented by a density matrix $\rho$ which is a positive semidefinite, self-adjoint matrix with $\trace(\rho) = 1$. The von-Neumann entropy is a well-known measure of quantum information which is given by $-\trace(\rho \log(\rho))$ for a density matrix $\rho$. Given two density matrices $\rho$ and $\sigma$ the quantum relative entropy is determined by $D(\rho || \sigma) = \trace \left[\rho (\log(\rho) - \log(\sigma)) \right]$.

Considering $\trace(A) = \trace(B) = 1$ in equation (\ref{inequality_bebiano_2003_matrix}), we find that $A$ and $B$ are density matrices. Then, the quantum relative entropy given two density matrices $A$ and $B$,
\begin{equation*} 
D(A||B) = \trace \left[A (\log(A) - \log(B)) \right] \geq 0.
\end{equation*} 
Considering $B = I$, the identity matrix of order $n$, in equation (\ref{matrix_inequality_with x_r_and_log_x}) we observe that
\begin{equation*}
\trace[A^r \log A^r] \geq \trace(A^r)[\log(\trace(A^r)) - \log(n)].
\end{equation*}
Putting $r = 1$ we have $\trace[A \log A] \geq \trace(A)[\log(\trace(A)) - \log(n)]$. In addition, if $A$ is a density matrix, that is, $\trace(A) = 1$ then the above inequality indicates $\trace[A \log A] \geq -\log(n)$. The von-Neumann entropy of $A$ is $- \trace[A \log A] \leq \log(n)$, which is its maximum value.

\begin{example} 
	The matrix counterpart of equation (\ref{q_deformed_log_sum_with_r}) will be given by
	\begin{equation*}
		\left(\trace(B^r)\right)^{1 - q} \trace\left[A^r \ln_q \left(A^r B^{-r} \right)\right] \geq \left(\trace(A^r)\right) \left[\ln_q \left(\trace(A^r)\right) - \ln_q \left(\trace(B^r)\right)\right],
	\end{equation*}
	where $A$ and $B$ are self-adjoint commuting matrices as well as $B$ is positive definite. For $r = 1$ we have
	\begin{equation*}
		\left(\trace(B)\right)^{1 - q} \trace\left[A \ln_q \left(A B^{-1} \right)\right] \geq \left(\trace(A)\right) \left[\ln_q \left(\trace(A)\right) - \ln_q \left(\trace(B)\right)\right].
	\end{equation*}
	Note that in the above two inequalities we consider $q < 2$ to ensure convexity of $x\ln_q(x)$. For $q > 2$ equation (\ref{q_deformed_log_sum_with_r_and_q_>_2}) suggests
	\begin{equation*}
		\left(\trace(B^r)\right)^{1 - q} \trace\left[A^r \ln_q \left(A^r B^{-r} \right)\right] \leq \left(\trace(A^r)\right) \left[\ln_q \left(\trace(A^r)\right) - \ln_q \left(\trace(B^r)\right)\right],
	\end{equation*} 
	for self-adjoint commuting matrices $A$ and $B$ as well as positive definite matrix $B$.
\end{example} 

Theorem \ref{reverse_log_sum_inequality_with_real_functions} also indicates a trace inequality, which we mention below without a proof.

\begin{theorem}
	Let $A$ and $B$ be commuting self-adjoint matrices, with the sets of eigenvalues $\Lambda(A) = \{a_1, a_2, \cdots , a_n\}$ and $\Lambda(B) = \{b_1, b_2, \cdots, b_n\}$, respectively. Also, let $g: \Lambda(A) \cup \Lambda(B) \rightarrow \mathbb{R}$ be a function, such that, $g(b_i) > 0$ for all $i = 1, 2, \cdots n$. In addition, $f : \left[\widetilde{m}_g, \widetilde{M}_g \right] \rightarrow \mathbb{R}$ is a function for which $h(x) = x f\left( \dfrac{1}{x} \right)$ is a concave. Then,
	$$\trace[g(A) f(g(A)(g(B))^{-1})] \leq \trace[g(A)] f \left(\frac{\trace(g(A))}{\trace(g(B))}\right).$$
\end{theorem}

\section{log-sum inequalities with L\"{o}wner partial order}

Recall that given self-adjoint matrices $A$ and $B$ we write $A \geq B$ or $B \leq A$ if the matrix $A - B$ is positive semidefinite. This ordering is called the L\"{o}wner partial order. If $A$ is positive definite we write $A > 0$. 

Let $M_n(J)$ be the set of all complex matrices of order $n$ with eigenvalues in $J \subset \mathbb{R}$. Equation (\ref{matrix_function}) indicates that we can redefine a function $f:J \rightarrow \mathbb{R}$ as a matrix function $f: M_n(J) \rightarrow M_n(\mathbb{C})$, the set of all complex matrices of order $n$. Let $A$ and $B$ be two self-adjoint matrices of order $n$ with eigenvalues in $J$. Now, a continuous function $f: J \rightarrow \mathbb{R}$  is said to be matrix monotone of order $n$ if $A \leq B$ implies $f(A) \leq f (B)$. Also, $f$ is said to be a matrix convex function of order $n$ if 
\begin{equation*}
f(\lambda A + (1 - \lambda) B) \leq \lambda f(A) + (1 - \lambda)f(B),
\end{equation*} 	
for all $\lambda \in [0, 1]$ and every pair of matrices $A$, and $B$ of order $n$. An operator monotone function is a matrix monotone function for any natural number $n$. Similarly, an operator convex function is a matrix convex function for any natural number $n$. A function $f$ is operator concave if $-f$ is operator convex. Well known operator monotone functions are $f_1 (t) = t^r$ for $0 < r \leq 1$ for $t \in [0, \infty)$, as well as $f_2(t) = \log(t)$ for $t \in (0, \infty)$. In addition, $g(t) = t^r$ is operator convex on $(0, \infty)$ for $-1 \leq r \leq 0$ and $1 \leq r \leq 2$ \cite{zhan2004matrix,bhatia2013matrix}. The operator Jensen inequality \cite{hansen1982jensen,hansen2003jensen} plays a crucial role for further development, which we mention below: 
\begin{lemma}{\bf (\cite{hansen1982jensen})}\label{hansen_jensen_theory}
	If $f$ is a continuous, real function defined on an interval $[0, \alpha)$ with $\alpha \leq \infty$, the following conditions are equivalent. 
	\begin{enumerate}
		\item 
		$f$ is operator convex and $f(0) \leq 0$.
		\item 
		$f(A^\dagger XA) \leq A^\dagger f(X) A$ for all $A$ with $||A|| \leq 1$ and every self-adjoint $X$ with spectrum in $[0, \alpha)$.
		\item 
		$f(A^\dagger XA + B^\dagger YB) \leq A^\dagger f(X) A + B^\dagger f(Y)B $ for all $A, B$ with $A^\dagger A + B^\dagger B \leq I$ and all $X, Y$ with spectrum in $[0, \alpha)$.
		\item 
		$f(PXP) \leq Pf(X) P$ for every projection $P$ and every self-adjoint $X$ with spectrum in $[0, \alpha)$.
	\end{enumerate}
\end{lemma}

We utilize Dirac's bra-ket notation \cite{dirac1939new} for simplifying our notations. Recall that $\ket{u}$ denotes a column vector of length $n$, or equivalently an $n \times 1$ matrix. The conjugate transpose of $\ket{u}$ is given by $\bra{u}$, which is a row vector of length $n$. Note that, $\ket{u} \bra{u}$ is a self-adjoint matrix of order $n$. Now the spectral decomposition of any self-adjoint matrix $A$ can be written as $A = \sum\limits_{i = 1}^n \lambda_i \ket{u_i} \bra{u_i}$ where $\ket{u_i}$ is a normalized eigenvector corresponding to the eigenvalue $\lambda_i$.

The superadditivity of a mean has been shown in \cite{KA1980} and the corresponding advanced results for a mean and a perspective have been studied in \cite{Dra2017,Nik2021}. The following proposition is proven by the elementary calculations under a certain condition.
\begin{proposition}\label{sec4_theorem6}
	Let $A_1, A_2, \cdots A_m$ and $B_1, B_2, \cdots B_m$ be a two sets of positive definite matrices of order $n$, and $A \geq m I$, where $\sum\limits_{i = 1}^m A_i=:A$. Also, let $f$ be an operator concave function on an interval $[0, \alpha)$ containing the eigenvalues of $B_i$ and $B = \sum\limits_{i = 1}^m B_i$. Then
	$$\sum_{i = 1}^m A_i^{\frac{1}{2}} f \left(A_i^{- \frac{1}{2}} B_i A_i^{-\frac{1}{2}} \right) A_i^{\frac{1}{2}} \leq A^{\frac{1}{2}}f \left( A^{-\frac{1}{2}} B A^{-\frac{1}{2}} \right) A^{\frac{1}{2}}.$$
\end{proposition}
\begin{proof}
	As each $B_i$ is a self-adjoint operator, the spectral decomposition of $B_i$ can be written as
	\begin{equation*}
	B_i = \sum_{k=1}^n \lambda_k^{(i)} \ket{u_k^{(i)}} \bra{u_k^{(i)}},
	\end{equation*}
	where $\ket{u_k^{(i)}}$ is an eigenvector associated to the eigenvalue $\lambda_k^{(i)}$ of $B_i$. 
	Then we have
	\begin{equation*} 
	A_i^{-1/2}B_iA_i^{-1/2}=\sum_{k=1}^n \lambda_k^{(i)}|w_k^{(i)}\rangle\langle w_k^{(i)}|,
	\end{equation*} 
	where $|w_k^{(i)}\rangle:=A_i^{-1/2}|u_k^{(i)}\rangle$. Now for a function $f$ we have 
	\begin{equation*}
	f\left(A_i^{-1/2}B_iA_i^{-1/2}\right)=\sum_{k=1}^n f\left(\lambda_k^{(i)}\right)|w_k^{(i)}\rangle\langle w_k^{(i)}|.
	\end{equation*} 
	Combining we get
	\begin{equation*}
	\begin{split} 
	A_i^{1/2}f\left(A_i^{-1/2}B_iA_i^{-1/2}\right)A_i^{1/2} = & \sum_{k=1}^n f\left(\lambda_k^{(i)}\right)A_i^{1/2}|w_k^{(i)}\rangle\langle w_k^{(i)}|A_i^{1/2}\\
	= & \sum_{k=1}^n f\left(\lambda_k^{(i)}\right)|u_k^{(i)}\rangle\langle u_k^{(i)}|=f\left(B_i\right).
	\end{split} 
	\end{equation*}
	Summing over $i$ we find 
	\begin{equation}\label{sum_f(B_i)_over_i}
	\sum_{i = 1}^m A_i^{1/2}f\left(A_i^{-1/2}B_iA_i^{-1/2}\right)A_i^{1/2} = \sum_{i = 1}^m f(B_i).
	\end{equation}
	If $A \geq m I$, then we have $\sum_{i = 1}^m A^{-\frac{1}{2}} A^{-\frac{1}{2}} =mA^{-1} \leq I$. Now, applying Lemma \ref{hansen_jensen_theory} we find that
	\begin{equation*}
	\begin{split} 
	A^{-\frac{1}{2}} \sum_{i = 1}^m f(B_i) A^{-\frac{1}{2}} & = \sum_{i = 1}^m A^{-\frac{1}{2}} f(B_i) A^{-\frac{1}{2}} \leq f \left(\sum_{i = 1}^m A^{-\frac{1}{2}} B_i A^{-\frac{1}{2}} \right) \\
	& = f \left(A^{-\frac{1}{2}} \sum_{i = 1}^m B_i A^{-\frac{1}{2}} \right) = f \left(A^{-\frac{1}{2}} B A^{-\frac{1}{2}} \right).
	\end{split} 
	\end{equation*}
	Multiplying $A^{\frac{1}{2}}$ in both side of the above ineuality we get
	\begin{equation*}
	\sum_{i = 1}^m f(B_i) \leq A^{\frac{1}{2}} f \left((A^{-\frac{1}{2}})^\dagger B A^{-\frac{1}{2}} \right) A^{\frac{1}{2}}.
	\end{equation*}
	Putting the value of $f(B_i)$ from equation (\ref{sum_f(B_i)_over_i}) we find the result. 
\end{proof}		

\begin{remark}
	If $A_i$ and $B_i$ are replaced by positive real numbers $a_i$ and $b_i$ respectively in the above theorem, we get
	\begin{equation*}
	\sum_{i = 1}^m a_i^{\frac{1}{2}} f \left(a_i^{- \frac{1}{2}} b_i a_i^{-\frac{1}{2}} \right) a_i^{\frac{1}{2}} \leq a^{\frac{1}{2}}f \left( a^{-\frac{1}{2}} b a^{-\frac{1}{2}} \right) a^{\frac{1}{2}},
	\end{equation*}
	where $a = \sum\limits_{i = 1}^m a_i$ and $b = \sum\limits_{i = 1}^m b_i$ as well as $f$ is a convex function. Simplifying we get
	\begin{equation}\label{remark1_eq01}
	\sum_{i = 1}^m a_i f\left(\frac{b_i}{a_i}\right) \leq a f \left(\frac{b}{a}\right).
	\end{equation}
	Considering $f(x) = \log(x)$ we get the usual log-sum inequality.
\end{remark}

\begin{corollary}\label{sec4_corollary2}
	Let $A_1, A_2, \cdots, A_m$ and $B_1, B_2, \cdots, B_m$ be two sets of positive definite self-adjoint operators with $A = \sum\limits_{i = 1}^m A_i$ and $B = \sum\limits_{i = 1}^m B_i$, such that, $A = B$. Also, let $f$ be an operator concave function on an interval $[0, \alpha)$ containing the eigenvalues of $B_i$ and $B$ as well as $f(1) = 0$. Then
	$$\sum_{i = 1}^m A_i^{\frac{1}{2}} f(A_i^{-\frac{1}{2}} B_i A_i^{-\frac{1}{2}}) A_i^{\frac{1}{2}} \leq 0.$$
\end{corollary}
\begin{proof}
	The proof follows trivially from Proposition \ref{sec4_theorem6}. 
\end{proof}

\begin{remark}
	The operator Shannon inequality was given in \cite{furuta2004parametric} 
	\begin{equation*}
	\sum_{i=1}^m A_{i}^{\frac{1}{2}} \log \left(A_i^{-\frac{1}{2}}B_iA_i^{-\frac{1}{2}}\right)A_i^{\frac{1}{2}}\le 0
	\end{equation*}
	under the assumption $\sum\limits_{i=1}^m A_i = \sum\limits_{i=1}^n B_i =I$.
	Our condition in Corollary \ref{sec4_corollary2} is slightly weaker than this assumption. We also observe that the operator Shannon inequality holds for any operator concave function $f$.
\end{remark}

If every $A_i$ is expansive (i.e., $A_i \geq I$), then the condition $A \geq mI$ in Proposition \ref{sec4_theorem6} is satisfied. However, we have not obtained a proper result for contractive condition such as $A_i \leq I$. Closing this section, we present a result which does not impose an additional condition on the matrices $A_i$. We need the following known facts for proving the next theorem:
\begin{lemma}{\bf (\cite{MR563403})}  \label{lemma8}
	Let $X$ and $A$ be bounded linear operators on a Hilbert space $\mathcal{H}$. Suppose that $X \geq 0$ and $||A|| \leq 1$. If $f$ is an operator monotone function defined on $[0, \infty)$, then $A^\dagger f (X) A \leq f(A^\dagger X A)$.
\end{lemma}
\begin{lemma}{\bf (\cite[p.14]{zhan2004matrix})} \label{lemma9}
	For any square matrix $X_i$ and positive definite matrices $A_i$ we have 
	\begin{equation*}
	\sum_{i = 1}^m X^\dagger_i A_i^{-1} X_i \geq \left(\sum_{i = 1}^m X_i\right)^\dagger \left( \sum_{i = 1}^m A_i \right)^{-1} \left(\sum_{i = 1}^m X_i\right).
	\end{equation*}
\end{lemma}	

\begin{theorem}\label{theorem10}
	Let $A_1, A_2, \cdots, A_m$ and $B_1, B_2, \cdots, B_m$ be two sets of positive definite matrices, as well as $A = \sum\limits_{i = 1}^m A_i$ and $B = \sum\limits_{i = 1}^m B_i$. Also, $f$ is an operator monotone function defined on $[0, \infty)$. Then we have
	\begin{equation}\label{theorem10_ineq01}
	\frac{1}{m} \left( \sum_{i = 1}^m B_i^{\frac{1}{2}} \right) \left[ \sum_{i = 1}^m f \left( B_i^{\frac{1}{2}} A_i^{-1} B_i^{\frac{1}{2}} \right) \right]^{-1} \left( \sum_{i = 1}^m B_i^{\frac{1}{2}} \right) \leq B^{\frac{1}{2}} \left[ f \left( B^{\frac{1}{2}}  A^{-1} B^{\frac{1}{2}} \right) \right]^{-1} B^{\frac{1}{2}},
	\end{equation}
	and
	\begin{eqnarray}
		&& \sum_{i = 1}^m A_i^{\frac{1}{2}} B_i^{\frac{1}{2}} \left[ f \left( B_i^{\frac{1}{2}} A_i^{-1} B_i^{\frac{1}{2}} \right) \right]^{-1} B_i^{\frac{1}{2}} A_i^{\frac{1}{2}} \nonumber \\
		&& \leq \frac{1}{m} \left( \sum_{i = 1}^m A_i^{\frac{1}{2}}\right) B^{\frac{1}{2}} \left[ f \left( B^{\frac{1}{2}}  A^{-1} B^{\frac{1}{2}} \right)\right]^{-1} B^{\frac{1}{2}} \left(\sum_{i = 1}^m A_i^{\frac{1}{2}} \right).\label{theorem10_ineq02}
	\end{eqnarray}
\end{theorem}

\begin{proof}
	We can write
	\begin{equation*}
	B_i^{\frac{1}{2}} A_i^{-1} B_i^{\frac{1}{2}} = B_i^{\frac{1}{2}} \left(B^{-\frac{1}{2}} B^{\frac{1}{2}} \right) A_i^{-1} \left(B^{\frac{1}{2}} B^{-\frac{1}{2}} \right) B_i^{\frac{1}{2}} = B_i^{\frac{1}{2}} B^{-\frac{1}{2}} \left( B^{\frac{1}{2}}  A_i^{-1} B^{\frac{1}{2}} \right) B^{-\frac{1}{2}} B_i^{\frac{1}{2}}.
	\end{equation*}
	We have $A \geq A_i$, that is $A^{-1} \leq A_i^{-1}$. Also for any complex square matrix $X$ we have $X^\dagger A^{-1} X \leq X^\dagger A_i^{-1} X$. Applying these together we obtain
	\begin{equation*}
	B_i^{\frac{1}{2}} A_i^{-1} B_i^{\frac{1}{2}} \geq B_i^{\frac{1}{2}} B^{-\frac{1}{2}} \left( B^{\frac{1}{2}}  A^{-1} B^{\frac{1}{2}} \right) B^{-\frac{1}{2}} B_i^{\frac{1}{2}}.
	\end{equation*}
	As $f(t)$ is a matrix monotone function, we can write
	\begin{equation*}
	f \left( B_i^{\frac{1}{2}} A_i^{-1} B_i^{\frac{1}{2}} \right) \geq f \left( B_i^{\frac{1}{2}} B^{-\frac{1}{2}} \left( B^{\frac{1}{2}}  A^{-1} B^{\frac{1}{2}} \right) B^{-\frac{1}{2}} B_i^{\frac{1}{2}} \right).
	\end{equation*}
	From $B > B_i$ for all $i = 1, 2, \cdots, n$,  we can see 
	\begin{equation*}
	I > B^{-\frac12}B_iB^{-\frac12} = \left(B_i^{\frac12}B^{-\frac12}\right)^{\dagger}B_i^{\frac12}B^{-\frac12}
	\end{equation*} 
	which implies $1 > \vert|\left(B_i^{\frac12}B^{-\frac12}\right)^{\dagger}B_i^{\frac12}B^{-\frac12} \vert| =\vert| B_i^{\frac12}B^{-\frac12}\vert|^2$, since $\vert| A\vert|=\vert|A^{\dagger}A\vert|^{\frac12}$ for every operator $A$ in general. Thus we have $\vert| B_i^{\frac12}B^{-\frac12}\vert| < 1$. We also have $\Vert B^{-\frac{1}{2}} B_i^{\frac{1}{2}} \Vert \leq 1$ so that we have the following inequality by Lemma \ref{lemma8},
	\begin{equation*}
	f \left( B_i^{\frac{1}{2}} B^{-\frac{1}{2}} \left( B^{\frac{1}{2}}  A^{-1} B^{\frac{1}{2}} \right) B^{-\frac{1}{2}} B_i^{\frac{1}{2}} \right) \geq \left( B_i^{\frac{1}{2}} B^{-\frac{1}{2}} \right) f \left( B^{\frac{1}{2}}  A^{-1} B^{\frac{1}{2}} \right) \left( B^{-\frac{1}{2}} B_i^{\frac{1}{2}} \right).
	\end{equation*}
	That is,
	\begin{equation*}
	f \left( B_i^{\frac{1}{2}} A_i^{-1} B_i^{\frac{1}{2}} \right) \geq \left( B_i^{\frac{1}{2}} B^{-\frac{1}{2}} \right) f \left( B^{\frac{1}{2}}  A^{-1} B^{\frac{1}{2}} \right) \left( B^{-\frac{1}{2}} B_i^{\frac{1}{2}} \right). 
	\end{equation*}
	Multiplying $B_i^{-\frac{1}{2}}$ to the both sides, we have
	\begin{equation*}
	B_i^{-\frac{1}{2}} f \left( B_i^{\frac{1}{2}} A_i^{-1} B_i^{\frac{1}{2}} \right) B_i^{-\frac{1}{2}} \geq B^{-\frac{1}{2}} f \left( B^{\frac{1}{2}}  A^{-1} B^{\frac{1}{2}} \right) B^{-\frac{1}{2}}.
	\end{equation*}
	Taking an inverse of the both sides, we have
	\begin{equation}\label{theorem10_eq57}
	B_i^{\frac{1}{2}} \left[ f \left( B_i^{\frac{1}{2}} A_i^{-1} B_i^{\frac{1}{2}} \right) \right]^{-1} B_i^{\frac{1}{2}} \leq B^{\frac{1}{2}} \left[ f \left( B^{\frac{1}{2}}  A^{-1} B^{\frac{1}{2}} \right) \right]^{-1} B^{\frac{1}{2}}.		
	\end{equation}
	Thus we have
	\begin{equation}\label{theorem10_eq58}
	\frac{1}{m} \sum_{i = 1}^m B_i^{\frac{1}{2}} \left[ f \left( B_i^{\frac{1}{2}} A_i^{-1} B_i^{\frac{1}{2}} \right) \right]^{-1} B_i^{\frac{1}{2}} \leq B^{\frac{1}{2}} \left[ f \left( B^{\frac{1}{2}}  A^{-1} B^{\frac{1}{2}} \right) \right]^{-1} B^{\frac{1}{2}}.
	\end{equation}
	Applying Lemma \ref{lemma9} to the above, we have
	\begin{eqnarray}
		&& \frac{1}{m} \sum_{i = 1}^m B_i^{\frac{1}{2}} \left[ f \left( B_i^{\frac{1}{2}} A_i^{-1} B_i^{\frac{1}{2}} \right) \right]^{-1} B_i^{\frac{1}{2}} \nonumber \\
		&& \geq \frac{1}{m} \left( \sum_{i = 1}^m B_i^{\frac{1}{2}} \right) \left[ \sum_{i = 1}^m f \left( B_i^{\frac{1}{2}} A_i^{-1} B_i^{\frac{1}{2}} \right) \right]^{-1} \left( \sum_{i = 1}^m B_i^{\frac{1}{2}} \right). \label{theorem10_eq59}
	\end{eqnarray}
	Combining \eqref{theorem10_eq58} and \eqref{theorem10_eq59}, we get the inequality \eqref{theorem10_ineq01}.
	
	To prove the inequality \eqref{theorem10_ineq02}, we start from \eqref{theorem10_eq57}. By multiplying $A_i^{\frac{1}{2}}$ to the both sides in \eqref{theorem10_eq57}, we have
	\begin{equation*}
	A_i^{\frac{1}{2}} B_i^{\frac{1}{2}} \left[ f \left( B_i^{\frac{1}{2}} A_i^{-1} B_i^{\frac{1}{2}} \right) \right]^{-1} B_i^{\frac{1}{2}} A_i^{\frac{1}{2}} \leq A_i^{\frac{1}{2}} B^{\frac{1}{2}} \left[ f \left( B^{\frac{1}{2}}  A^{-1} B^{\frac{1}{2}} \right) \right]^{-1} B^{\frac{1}{2}} A_i^{\frac{1}{2}}.
	\end{equation*}
	Taking a summation on $i$ from $1$ to $m$ for the both sides in the above inequality, we obtain
	\begin{equation*}
	\sum_{i = 1}^m A_i^{\frac{1}{2}} B_i^{\frac{1}{2}} \left[ f \left( B_i^{\frac{1}{2}} A_i^{-1} B_i^{\frac{1}{2}} \right) \right]^{-1} B_i^{\frac{1}{2}} A_i^{\frac{1}{2}} \leq \sum_{i = 1}^m A_i^{\frac{1}{2}} B^{\frac{1}{2}} \left[ f \left( B^{\frac{1}{2}}  A^{-1} B^{\frac{1}{2}} \right) \right]^{-1} B^{\frac{1}{2}} A_i^{\frac{1}{2}}.
	\end{equation*}
	That is, we have
	\begin{eqnarray}
&&		\sum_{i = 1}^m A_i^{\frac{1}{2}} B_i^{\frac{1}{2}} \left[ (-1) f \left( B_i^{\frac{1}{2}} A_i^{-1} B_i^{\frac{1}{2}} \right) \right]^{-1} B_i^{\frac{1}{2}} A_i^{\frac{1}{2}} \nonumber \\
&& \geq \sum_{i = 1}^m A_i^{\frac{1}{2}} B^{\frac{1}{2}} \left[(-1) f \left( B^{\frac{1}{2}}  A^{-1} B^{\frac{1}{2}} \right) \right]^{-1} B^{\frac{1}{2}} A_i^{\frac{1}{2}}.\label{theorem10_eq62}
	\end{eqnarray}			
	Applying Lemma \ref{lemma9} to the right hand side in \eqref{theorem10_eq62}, we have
	\begin{equation*} 
	\begin{split}
	& \sum_{i = 1}^m A_i^{\frac{1}{2}} B^{\frac{1}{2}} \left[(-1) f \left( B^{\frac{1}{2}}  A^{-1} B^{\frac{1}{2}} \right) \right]^{-1} B^{\frac{1}{2}} A_i^{\frac{1}{2}} \\
	& \geq \left( \sum_{i = 1}^m A_i^{\frac{1}{2}}\right) \left( (-m) B^{-\frac{1}{2}} f \left( B^{\frac{1}{2}}  A^{-1} B^{\frac{1}{2}} \right) B^{-\frac{1}{2}} \right)^{-1} \left(\sum_{i = 1}^m A_i^{\frac{1}{2}} \right),
	\end{split} 
	\end{equation*} 
	that is,
	\begin{equation*}
	\begin{split}
	& - \sum_{i = 1}^m A_i^{\frac{1}{2}} B^{\frac{1}{2}} \left[f \left( B^{\frac{1}{2}}  A^{-1} B^{\frac{1}{2}} \right) \right]^{-1} B^{\frac{1}{2}} A_i^{\frac{1}{2}} \\
	& \geq (- 1) \left( \sum_{i = 1}^m A_i^{\frac{1}{2}}\right) \frac{1}{m}  B^{\frac{1}{2}} \left[ f \left( B^{\frac{1}{2}}  A^{-1} B^{\frac{1}{2}} \right)\right]^{-1} B^{\frac{1}{2}} \left(\sum_{i = 1}^m A_i^{\frac{1}{2}} \right),
	\end{split}
	\end{equation*}
	which implies
	\begin{eqnarray}
&&	\sum_{i = 1}^m A_i^{\frac{1}{2}} B^{\frac{1}{2}} \left[f \left( B^{\frac{1}{2}}  A^{-1} B^{\frac{1}{2}} \right) \right]^{-1} B^{\frac{1}{2}} A_i^{\frac{1}{2}} \nonumber \\
&& \leq \frac{1}{m} \left( \sum_{i = 1}^m A_i^{\frac{1}{2}}\right) B^{\frac{1}{2}} \left[ f \left( B^{\frac{1}{2}}  A^{-1} B^{\frac{1}{2}} \right)\right]^{-1} B^{\frac{1}{2}} \left(\sum_{i = 1}^m A_i^{\frac{1}{2}} \right). \label{theorem10_eq64}
	\end{eqnarray}
	Combining \eqref{theorem10_eq62} and  \eqref{theorem10_eq64}, we get the inequality \eqref{theorem10_ineq02}. 
\end{proof}

\begin{remark}
	Considering $a_i$ and $b_i$ are positive real numbers and replacing $A_i =: a_i$ and $B_i =: b_i$ as well as $f$ is a monotone increasing function we have from equation (\ref{theorem10_ineq01})
	\begin{equation}\label{remark3_eq01}
	\frac{1}{m} \left( \sum_{i = 1}^m b_i^{\frac{1}{2}} \right)^2 \left[ \sum_{i = 1}^m f \left( \frac{b_i}{a_i} \right) \right]^{-1} \leq b \left[ f \left( \frac{b}{a} \right) \right]^{-1}.
	\end{equation}
	Also, from equation (\ref{theorem10_ineq02}) we can write
	\begin{equation}\label{remark3_eq02}
	\sum_{i = 1}^m a_i b_i \left[ f \left( \frac{b_i}{a_i} \right) \right]^{-1} \leq \frac{1}{m} \left( \sum_{i = 1}^m a_i^{\frac{1}{2}}\right)^2 b \left[ f \left( \frac{b}{a} \right)\right]^{-1}.
	\end{equation}
\end{remark}

\section{Conclusion}

The log-sum inequality which is mentioned in equation (\ref{log_sum_inequality}) plays a crucial role in classical and quantum information theory. In this article, we present a number of analogous inequalities. Inequality (\ref{log_sum_inequality}) depends on the convexity of the function $x\log(x)$ which was relaxed utilizing the function $xf(x)$, in \cite{csiszar2004information}. We illustrate that the concavity of $x f(\dfrac{1}{x})$ is also efficient for deriving similar inequalities. We discuss these inequalities with two functions instead of the single function $f$, which increase their scope of applications. In this context, the log-sum inequality for $q$-deformed logarithm is illustrated. We elaborate these results for the commuting self-adjoint matrices as trace-form inequalities. Later we discuss the log-sum inequality for the operator monotone and convex functions in the context of L\"{o}wner partial order relation of the positive semi-definite matrices, which is an application of the Hansen operator inequality. The following problem may be attempted in future.

Both inequalities in \eqref{remark3_eq01} and \eqref{remark3_eq02} do not have same form of log-sum inequality \eqref{remark1_eq01}. Therefore we have to conclude that the restricted condition (so-called expansivity of matrices $A_i$) given in Proposition \ref{sec4_theorem6} leads us to obtain the log-sum inequality for non-commutative matrices. However, we could not obtain such an inequality without restricted condition in Theorem \ref{theorem10}. In the further studies on this topic, we would like to obtain the log-sum inequality for non-commutative contractive matrices. We hope that the obtained results will be useful to study the parametric information theory in the future.

\section*{Funding}
SD was a post doctoral fellow at the Indian Institute of Technology Kharagpur during this work.
SF was partially supported by JSPS KAKENHI grant number 21K03341. 


\end{document}